\documentclass[12pt,fleqn]{amsart}

\usepackage{amssymb}

\newtheorem{theorem}{Theorem}[section]
\newtheorem{lemma}[theorem]{Lemma}

\newtheorem{problem}[theorem]{Problem}
\newtheorem{corollary}[theorem]{Corollary}
\theoremstyle{definition}

\theoremstyle{remark}

\numberwithin{equation}{section}

\begin{document}

\title[Isomorphisms of  $C(K)$ spaces]{On isomorphisms of Banach spaces\\
of continuous functions}

\author[G.\ Plebanek]{Grzegorz Plebanek}
\address{Instytut Matematyczny, Uniwersytet Wroc\l awski}
\email{grzes@math.uni.wroc.pl}


\thanks{The  research partially suported by  by MNiSW Grant N N201 418939 (2010--2013).}

\thanks{I would like to thank Eloi Medina Galego, Miko{\l}aj Krupski and Witold Marciszewski for the correspondence
concerning the subject.}

\begin{abstract}
We prove that if $K$ and $L$ are compact spaces and $C(K)$ and $C(L)$ are
isomorphic as Banach spaces then $K$ has a $\pi$-base consisting of open sets
$U$ such that $\overline{U}$ is a continuous image of some compact subspace of $L$.
This sheds a new light on isomorphic classes of the spaces of the form $C([0,1]^\kappa)$ and spaces $C(K)$ where $K$ is Corson compact.
\end{abstract}

\maketitle

\newcommand{\con}{\mathfrak c}
\newcommand{\eps}{\varepsilon}
\newcommand{\alg}{\mathfrak A}
\newcommand{\algb}{\mathfrak B}
\newcommand{\algc}{\mathfrak C}
\newcommand{\ma}{\mathfrak M}
\newcommand{\pa}{\mathfrak P}
\newcommand{\BB}{\protect{\mathcal B}}
\newcommand{\AAA}{\mathcal A}
\newcommand{\CC}{{\mathcal C}}
\newcommand{\FF}{{\mathcal F}}
\newcommand{\GG}{{\mathcal G}}
\newcommand{\LL}{{\mathcal L}}
\newcommand{\KK}{{\mathcal K}}
\newcommand{\UU}{{\mathcal U}}
\newcommand{\VV}{{\mathcal V}}
\newcommand{\HH}{{\mathcal H}}
\newcommand{\DD}{{\mathcal D}}
\newcommand{\RR}{\protect{\mathcal R}}
\newcommand{\ide}{\mathcal N}
\newcommand{\btu}{\bigtriangleup}
\newcommand{\hra}{\hookrightarrow}
\newcommand{\ve}{\vee}
\newcommand{\we}{\cdot}
\newcommand{\de}{\protect{\rm{\; d}}}
\newcommand{\er}{\mathbb R}
\newcommand{\qu}{\mathbb Q}
\newcommand{\supp}{{\rm supp} }
\newcommand{\card}{{\rm card} }
\newcommand{\wn}{{\rm int} }
\newcommand{\ult}{{\rm ULT}}
\newcommand{\vf}{\varphi}
\newcommand{\osc}{{\rm osc}}
\newcommand{\ol}{\overline}
\newcommand{\me}{\protect{\bf v}}
\newcommand{\ex}{\protect{\bf x}}
\newcommand{\stevo}{Todor\v{c}evi\'c} 
\newcommand{\cc}{\protect{\mathfrak C}}
\newcommand{\scc}{\protect{\mathfrak C^*}}
\newcommand{\lra}{\longrightarrow}
\newcommand{\sm}{\setminus}
\newcommand{\wh}{\widehat}

\newcommand{\uhr}{\upharpoonright}
\newcommand{\dx}{\;{\rm d}}
\newcommand{\sub}{\subseteq}
\newcommand{\m}{$(M)$} 
\newcommand{\MA}{MA$(\omega_1)$} 
\newcommand{\clop}{\protect{\rm Clop} }
\section{Introduction}

Let $K$ and $L$ be compact spaces such that $C(K)$ and $C(L)$ are isomorphic Banach  spaces ($C(K)\sim C(L)$).
When the spaces are isometric the classical Banach-Stone theorem   says that $K$ and $L$ are necessarily homeomorphic,
see e.g.\ Semadeni \cite{Se71}, 7.8.4.  Amir \cite {Am65} and Cambern \cite{Ca67} proved that this also holds when the isomorphism constant is
smaller then $2$. In this paper we study what can be said on the relation between $K$ and $L$ when the isomorphism constant is
arbitrary. Some results on the relations between $K$ and $L$ were also proved by Benyamini \cite{Be81} and Jarosz \cite{Ja84}
when $C(K)$ is only assumed to be isomorphic --- with small constant --- to a subspace of $C(L)$. We also study such relations for arbitrary embeddings.

The isomorphic types of $C(K)$ for metrizable compacta $K$ are known, see Bessaga and Pe\l czy\'{n}ski \cite{BP60} for countable
$K$ and Miljutin \cite{Mi66} in the uncountable case (see also Pe\l czy\'{n}ski \cite{Pe68}, Albiac and Kalton \cite{AK06} or Semadeni \cite{Se71}).
Outside the class of metric spaces the isomorphic classification of $C(K)$ spaces 
exists only in some special cases, see  Galego \cite{Ga09}.

The following open question has been around for several years, see e.g.\  3.9 in \cite{AMN},
6.45 in Negrepontis \cite{Ne84} or Question 1 in Koszmider \cite{Ko07}.

\begin{problem}\label{in:2}
Assume $C(K)\sim C(L)$ and that $L$ is a Corson compact space. Is $K$ necessarily Corson compact? 
\end{problem}

By the classical Amir-Lindenstrauss theory, 
the analogous question has a positive answer for the class of Eberlein compacta (weakly compact subsets of Banach spaces), 
since $K$ is Eberlein compact if and only if
$C(K)$ is weakly compactly generated  (see e.g.\ Negrepontis \cite{Ne84}). 
Recall also that  the answer to \ref{in:2} is positive assuming
Martin's axiom and the negation of continuum hypothesis (MA+$\neg$CH), see Argyros et.\ al.\ \cite{AMN}.

In \cite{Pl13} we studied similar  problems for positive isomorphisms or embeddings. This paper continues this study without the 
positivity assumption.
We now describe the organization of the paper. After preliminaries in section 2 we study embeddings in section 3. The main
result is Theorem \ref{ie:3} which says that if $T$ is an embedding of $C(K)$ into $C(L)$ then for every $x\in K$ there is
$y\in L$ such that the measure $T^*\delta_y$ has a large atom at $x$. As a corollary we obtain that under CH the space
$C(2^{\omega_1})$ is not isomorphic to a subspace of $L$ when $L$ is Corson compact. This has been already known 
under MA+$\neg$CH, see \cite{AMN}; see also \cite{MP10}.

 Given a surjective isomorphism $T: C(K)\to C(L)$,  we study in section 4 
the function $L\ni y\to ||T^*\delta_y||$ --- the main result is stated as Theorem \ref{i:4}.
Then in section 5 we analyse  
the set-valued functions that assign 
to every $y\in L$ the set of large atoms of $T^*\delta_y$. This leads to the main result of the paper, Theorem \ref{rq:1},
which says that that for every
nonempty open set $U\sub K$ there are an nonempty open set $V$ with $\ol{V}\sub U$ and a compact subset $L_1$ of $L$ 
such that $\ol{V}$ is a continuous image of $L_1$. As an application we obtain a partial positive solution
to Problem \ref{in:2}, namely, we 
prove that if $C(K)\sim C(L)$, where $L$ is Corson compact and $K$ is homogeneous, then $K$ is
also Corson compact.

\section{Preliminaries}
Let $K$ be a compact space. 
The dual space $C(K)^*$ of the Banach space $C(K)$ is identified with $M(K)$ --- the
space of all signed Radon measures of finite variation; we use the symbol $M_1(K)$ to denote the unit ball of $M(K)$.
Every $\mu\in M(K)$ can be written as $\mu=\mu^+-\mu^-$ where $\mu^+$ and $\mu^-$
are mutually singular nonnegative finite Radon measure. Recall that the variation 
$|\mu|$ of $\mu$ is defined as $|\mu|=\mu^+ + \mu^-$ and the natural norm in $M(K)$
is given by the formula  $||\mu||=|\mu|(K)$.

In the sequel, the space $M(K)$ is always equipped with the $weak^*$
topology inherited from $C(K)^*$, i.e.\ the topology making all the functionals
$\mu\to \int_K g\;{\rm d}\mu$ 
continuous,  where $g\in C(K)$. Note that we usually write $\mu(g)$ for $\int_K g\de\mu$. 
For
any $x\in K$ we denote by $\delta_x\in M(K)$  the corresponding Dirac measure.

The mapping $M(K)\ni\mu\to |\mu|\in M(K)$ is not $weak^*$ continuous; nonetheless   
it has the following semicontinuity properties.
Recall that a real-valued function $\vf$ (defined on some topological space $X$) is lower semicontinuous if the set $\{x\in X: \vf(x)>r\}$ is open for every $r\in\er$.

\begin{lemma}\label{pr:1}
For every compact space $K$
\begin{itemize}
\item[(i)] the mapping $M(K)\ni \mu\to |\mu|(g)$ is lower semicontinuous
for every nonnegative function $g\in C(K)$;
\item[(ii)] the mapping $M(K)\ni \mu\to |\mu|(U)$ is lower semicontinuous for every
open set $U\sub K$.
\end{itemize}
\end{lemma}

For the rest of the paper we fix two compacta  $K$ and $L$  such that $C(K)$ is embeddable into $C(L)$ and, unless
stated explicitly otherwise,
we constantly use  the following notation:  We fix  a linear operator $T:C(K)\to C(L)$ such that
\[m\cdot ||g||\le ||Tg||\le ||g||,\]
for all $g\in C(K)$, where $m>0$. For every $y\in L$ we write $\nu_y=T^*\delta_y$, i.e.\
$\nu_y\in M(K)$ is defined by the formula $\nu_y(g)=Tg(y)$ for $g\in C(K)$. 
Moreover, we write $\theta(y)=||\nu_y||$; thus $\theta$ is a real-valued (lower semicontinuous) function on $L$ 

Put $E=T[C(K)]$. For every
$x\in K$  we  let $\mu_x$ be any Hahn-Banach extension of $(T^{-1})^*\delta_x$, i.e.\
$\mu_x$ is defined on $E$ by $\mu_x(Tg)=g(x)$ and then extended to a functional on $C(L)$ with the same norm.

In section 5 we consider set-functions from $L$ into
$[K]^{<\omega}$, the family of all finite subsets of $K$.
Recall that a function $\vf:L\to [K]^{<\omega}$ is said to be
upper semicontinuous if the set $\{y\in L: \vf(y)\sub U\}$ is open for every open $U\sub K$.
For any set $Y\sub L$ we write
\[\vf[Y]=\bigcup_{y\in Y} \vf(y);\]
we say that $\vf $ is surjective if $\vf[L]=K$.

A compact space $K$ is {\em Corson compact} if, for some cardinal number $\kappa$, which can be taken to be equal to the topological weight of $K$,
$K$ is homeomorphic to a subset of the $\Sigma$--product of real lines
\[\Sigma(\er^\kappa)=\{x\in \er^\kappa: |\{\alpha: x_\alpha\neq 0\}|\le\omega\}.\]
The class of Corson compacta has been intensively studied for its interesting topological properties
and various connections to functional analysis; we refer the reader to a basic paper \cite{AMN}
by Argyros, Mercourakis and Negrepontis,  and to the extensive surveys by Negrepontis \cite{Ne84} and Kalenda \cite{Ka00}. Recall that the class of Corson compacta is stable under taking closed subspaces and continuous images.

If $h$ is a real-valued function on some topological space $X$ and $A\sub X$ then 
for $x\in A$ we denote by $\osc_x(h,A)$ the oscillation of $h$ at $x$ on the set $A$. i.e.\ 
\[\osc_x(h,A)=\inf_U\sup\{|h(x')-h(x'')|: x',x''\in U\cap A\},\]
where the infimum is taken over all open neighbourhoods $U$ of $x$.

\section{Isomorphic embeddings}

Lemma \ref{ie:1} was noted by Jarosz  \cite{Ja84}.

\begin{lemma}\label{ie:1}
If $\mu=\mu_x$ for some fixed $x\in K$ then $||\nu_y||\ge m$ for $\mu$-almost all $y\in L$.
\end{lemma}

\begin{proof}
Let $N=\{y\in L: ||{\delta_y}_{|E}||<1 \}$; then $\mu(N)=0$.

Indeed, we have $N=\bigcup_{r<1}N_r$, where the sets 
\[N_r=\{y\in L: ||{\delta_y}_{|E}||\le r \},\]
 are closed; it is therefore sufficient to check that $|\mu|(N_r)=0$ for every $r<1$.

Take any $\eps>0$ and $g\in C(K)$ such that $||Tg||\le 1$ and $\mu(Tg)> ||\mu||-\eps$.
Then
\[||\mu||-\eps=|\mu|(N_r)+|\mu|(L\sm N_r)-\eps< \mu(Tg)=\]
\[=\int_{N_r} Tg \dx\mu+\int_{L\sm N_r} Tg\dx\mu\le r|\mu|(N_r)+|\mu|(L\sm N_r),\]
which gives $|\mu|(N_r)\le \eps/(1-r)$ and, consequently, $|\mu|(N_r)=0$, as $\eps>0$ is arbitrary.

Now for any $y\in L\sm N$ and any $\eps>0$ there is
$g\in C(K)$, $||Tg||\le 1$ such that $|Tg(y)|> 1-\eps$. Then $||g||\le 1/m$ and 
$|\nu_y(g)|=|Tg(y)|>1-\eps$. This implies $||\nu_y||\ge m$, and we are done.
\end{proof}

\begin{lemma}\label{ie:2}
Consider  a fixed $x\in K$ and the measure $\mu=\mu_x$. 
Assume that $\eps>0$ and that there is a compact subset $F\sub L$ such that 

\begin{itemize}
\item[(i)] $||\nu_y||\ge m$ for every $y\in F$;
\item[(ii)] $\osc_y(\theta,F)\le \eps$ for every $y\in F$;
\item[(iii)]  $|\mu|(L\sm F)<\eps$.
\end{itemize}

Then there is $y\in F$ such that $|\nu_y\{x\}|\ge m-2\eps$.
\end{lemma}

\begin{proof} 
Let $H$ be any neighbourhood of $x$ and $f_H:K\to [0,1]$ be a continuous function 
such that $f_H(x)=1$ and $f_H=0$ outside $H$.
We shall check that there is $y_H\in F$ such that $Tf_H(y_H)\ge m-\eps$.

Indeed, otherwise $|\nu_y(f_H)|<m-\eps$ for $y\in  F $ which, together with $|\mu|( F )\le ||\mu||\le 1/m$
and $||T||=1$,  would give 
\[
1=f_H(x)=\mu(Tf_H)=
\int_{ F }Tf_H\dx\mu+\int_{L\sm  F } Tf_H\dx\mu < \]
\[< (m-\eps)|\mu|( F )+\eps\le\frac{m-\eps}{m}+\eps\le 1,\]
a contradiction. In particular, it follows that $|\nu_{y_H}|(H)\ge m-\eps$.
 
The net $y_H$ ordered by the reverse inclusion of the $H$'s has a converging subnet
$(y_H)_{H\in \HH}$; denote its limit by $y$. By (ii) we may assume that $||\nu_{y_H}||\le ||\nu_y||+\eps$
for every $H\in\HH$.

We shall prove that $|\nu_y|(\{x\})\ge m-2\eps$; by regularity it suffices to
check that $|\nu_y|(U)\ge m-2\eps$ for every open set $U\ni x$. 

Given such an open set $U\ni x$, choose a continuous function $g:K\to [0,1]$ such that
$g=0$ outside $U$ and $g=1$ on an open set $V$ containing $x$. For any $H\in\HH$ with $H\sub V$
we have $|\nu_{y_H}|(g)\ge |\nu_{y_H}|(H)\ge m-\eps$.
Hence
 \[|\nu_{y_H}|(1-g) \le |\nu_{y_H}|(K)-(m-\eps)\le |\nu_y|(K)+\eps-(m-\eps)=|\nu_y|(K)-m+2\eps.\]
Since $\nu_{y_H}\to\nu_y$, it follows from Lemma \ref{pr:1}(i) and the above inequality that
\[|\nu_y|(1-g)\le |\nu_y|(K)-m+2\eps.\]
 We conclude that $|\nu_y|(U)\ge |\nu_y|(g)\ge m-2\eps$ and the proof is complete.
\end{proof}

We are ready for the main result of this section.

\begin{theorem}\label{ie:3}
If $T:C(K)\to C(L)$ is an isomorphic embedding then, writing $\nu_y=T^*\delta_y$ for 
$y\in L$, for every $x\in K$
we have 
\[\sup\{|\nu_y(\{x\})|:y\in L\}\ge \frac{1}{||T||||T^{-1}||}.\]
\end{theorem}

\begin{proof}
Clearly we can assume that $||T||=1$ and denote $m=1/||T^{-1}||>0$.
The function $\theta:L\ni y\to ||\nu_y||$ is lower semicontinuous hence Borel. Given $x\in K$, 
by Lemma \ref{ie:1}  $||\nu_y||\ge m$ $\mu_x$-almost everywhere.
By the Lusin theorem we can therefore find for any  $\eps>0$  a compact set $ F \sub L$ 
with $|\mu|(L\sm F )<\eps$ and such that $\theta$ is continuous on $F$ and $\theta(y)\ge m$ for $y\in F$.
Applying Lemma \ref{ie:2} we finish  the proof.
\end{proof}

%
%
%

We conclude this section by showing the following result on Corson compacta. 

\begin{corollary}\label{ie:5}
Let $K$ be such a compact space that $\card K>\con = \card [C(K)]$ and that
$L$ is Corson compact.
Then $C(K)$ cannot be embedded into $C(L)$.
\end{corollary}

\begin{proof}
Suppose otherwise and let $T:C(K)\to C(L)$ be an embedding, where $L$ is Corson compact.
Since the class of Corson compacta is closed under taking continuous images, 
by passing to a quotient of $L$ we can additionally assume that the functions from $E=T[C(K)]$ distinguish points of $L$.
As the space $C(K)$ has cardinality $\con$, 
this implies that the topological weight of $L$ is at most $\con$. Thus $L$ is homeomorphic to a subspace of $\Sigma(\er^\con)$
so in particular $\card L\le \con$. 

On the other hand, the cardinality of the sets $\{ x\in K: |\nu_y|\{x\})\ge m/2\}$ is at most $2/m$ and they cover 
all of $K$ by Theorem \ref{ie:3}. It follows that $\card K\le\con$,
contrary to our assumption.
\end{proof}

\begin{corollary}\label{ie:6}
Assuming CH, $C(2^{\omega_1})$ cannot be embedded into $C(L)$ with $L$ being Corson compact.
\end{corollary}  

\begin{proof}
Under CH the Cantor cube $K=2^{\omega_1}$ has cardinality $2^\con>\con$. 
Moreover, there are only $\con$ many continuous functions on $K$ because every
such a function is determined by countably many coordinates. Hence
we can apply Corollary \ref{ie:5}. 
\end{proof}

Note that in Corollary \ref{ie:6} CH can be relaxed to $2^{\omega_1}>\con$.
At this point it is worth recalling that under MA+$\neg$CH, if $L$ is Corson compact then
$M_1(L)$ is also Corson compact in its $weak^*$ topology, Argyros et al. \cite{AMN}.
Consequently, if $T: C(K)\to C(L)$ is an embedding then $T^*[M_1(L)]$ is Corson compact and
so is $K$ ($T^*[M_1(L)]$ contains a ball in $M(K)$ because $T^*$ is onto and $K$ can be
embedded into the space of measures via the mapping $K\ni x\to \delta_x$).
We do not know if Corollary \ref{ie:6} can be proved without any extra set-theoretic
axioms, i.e.\ we do not know what happens when $2^{\omega_1}=\con$ but MA does not hold. It will become clear in section 6 that spaces such as $C(2^{\omega_1})$ cannot be isomorphic
to a space $C(L)$ whenever $L$ is Corson compact.

\section{Isomorphisms}

Keeping the notation from section 2, we shall now consider the case when $T:C(K)\to C(L)$ 
is an isomorphism. Note that in this case the measure $\mu_x\in M(L)$  is uniquely determined
by the condition $\mu_x(Tg)=g(x)$, $g\in C(K)$.

We start by the following general observation on lower semicontinuous (lsc) functions.

\begin{lemma}\label{i:12}
Let $f$ be a bounded lsc function on $K$, let $U$ be a nonempty  open subset of $K$ 
and fix $\eta>0$. Then there is a nonempty open set $V\sub U$ such that the oscillation of $f$ on $V$
is $\le\eta$. The same is true for a finite collection of bounded lsc functions or for differences
of such functions. 
\end{lemma}

\begin{proof}
By our assumption on $f$, the set $C_i=\{x\in K: f(x)\le i\eta\}$ is closed for every integer $i$. As $f$ is bounded there is minimal
$i$ such that $U\cap\wn{C_i}\neq\emptyset$. By minimality, $U\cap\wn{C_i}$ 
is not contained in $C_{i-1}$ so $V=(U\cap\wn{C_i})\sm C_{i-1}\neq\emptyset$.
The oscillation of $f$ is smaller than $\eta$ on $C_i\sm C_{i-1}$ so certainly on $V$.

Given bounded lsc $f_1,\ldots, f_k$ and an open set $U=V_0$ we just iterate the step above 
to find open sets $V_i\sub V_{i-1}$ such that the oscillation of $f_i$ is smaller than $\eta$.

If $f_j=g_j-g_j'$ where $g_j, g_j'$ are bounded lsc functions we can apply the above argument to $g_j$'s with $\eta/2$.
\end{proof}

\begin{lemma}\label{i:3}
Let $Y_1\sub Y_2\sub Y_3=L$ be closed subsets of $L$ and let 
$\eta>0$. Suppose that $U\sub K$ is an open neighbourhood of $x_0\in K$ such that for $j=1,2,3$
and every $x\in U$
\[(*)\quad  \big| |\mu_x|(Y_j) - |\mu_{x_0}|(Y_j) \big|<\eta.\]
Then there is an open set $V$ with $x_0\in  V\sub U$ and a compact set $L_1\sub Y_2\sm Y_1$ such that
for every $x\in V$ we have
\[|\mu_x|(L_1)>|\mu_x|(Y_2\sm Y_1)-4\eta.\]

More generally, if $Z_1\sub Z_2\sub \ldots \sub Z_N=L$ are closed sets, $\eta>0$ and
(*) holds for every $1\le j\le N$ then for any open set $U$ with
 $x_0\in U\sub K$ there are an open set $x_0\in V \sub U$ and compact sets $L_j\sub Z_{j+1}\sm Z_j$, $j\le N-1$ such that
\[|\mu_x|(L_j)>|\mu_x|(Z_{j+1}\sm Z_j)-4\eta,\]
for every $x\in V$.
\end{lemma}

\begin{proof}
Note that there is a continuous function $h:L\to [0,1]$ such that its support 
$S=\overline{\{y\in L: h(y)>0\}}$ is disjoint from $Y_1$ and
\[\mu_{x_0}(h)> |\mu_{x_0}|(L\sm Y_1)-\eta.\]
We put $L_1=S\cap Y_2$ and define $V$ as 
\[V=\{x\in U: \big|\mu_x(h)-\mu_{x_0}(h)\big|<\eta\};\]
then $V$ is open since the mapping $x\to \mu_x(h)$ is continuous.

Now  for any $x\in V$
\[|\mu_x|(S)\ge \mu_x(h)>\mu_{x_0}(h)-\eta>|\mu_{x_0}|(L\sm Y_1)-2\eta>|\mu_x|(L\sm Y_1)-4\eta,\]
where the last inequality follows from (*) applied to $l=1$ and $l=3$.
It follows that 
\[ |\mu_x|(L_1)=|\mu_x|(S\cap Y_2)\ge  |\mu_x|(Y_2\sm Y_1)-4\eta,\]
for $x\in V$, and this shows the first assertion. 

The second part follows by iteration of the first part to $Y_1=Z_j$, $Y_2=Z_{j+1}$ and with the resulting open sets
$U=V_0\supseteq V_1\supseteq\ldots V_{N-1}=V$.
\end{proof}

\begin{theorem}\label{i:4}
Fix  $\eps>0$. Then  for every nonempty open set $U\sub K$ there is a nonempty open set 
$V\sub U$ and a compact set $F\sub L$ such that

\begin{itemize}
\item[(a)] $\osc_y(\theta,F)\le \eps$ for every $y\in F$;
\item[(b)] for every $x\in V$ there is $y\in F$ such that $|\nu_y(\{x\})|\ge m-2\eps$.
\end{itemize}
\end{theorem}

\begin{proof}
We use Lemma \ref{ie:2}. Condition (i) of the lemma holds trivially for isomorphisms and we construct
$ F $ satisfying  \ref{ie:2}(ii)-(iii).

Let $Z_i=\{y\in L: ||\nu_y||\le m+\eps i\}$.
Then $Z_1\sub Z_2\sub \ldots \sub L$ are closed and $Z_{N}=L$ for some $N\le 1/\eps$.
Take $\eta>0$ such that $\eta=\eps/(4N)$.

The functions $x\to |\mu_x|(Z_i)=||\mu_x||-|\mu_x|(L\sm Z_i)$ are differences of lsc functions
so by Lemma \ref{i:12} there is a nonempty open set $W\sub U$ such that their oscillations
on $W$ are smaller than $\eta$.

By Lemma \ref{i:3} there are disjoint compact sets $L_i$ and a nonempty open set $V\sub W$ such that
\[|\mu_x|(L_i)> |\mu_x|(Z_{i+1}-Z_i)-4\eta,\]
for every $i$ and $x\in V$. It follows that, writing $L_0=Z_1$, the compact set $ F = \bigcup_{0\le i<N} L_i$ satisfies
$|\mu_x|( F )>|\mu_x|(L)-\eps$. As the oscillation of $y\to ||\nu_y||$ is smaller than $\eps$ on each of the closed disjoint
sets $L_0,\ldots, L_{N-1}$, (a) holds and now Lemma \ref{ie:3} gives (b).
\end{proof}

\section{Finite valued maps}

We shall consider now set-valued mappings related
to the  isomorphism $T:C(K)\to C(L)$. For any $r>0$ and $y\in L$ we define
\[\vf_r(y)=\{x\in K: |\nu_y(\{x\})|\ge r\}.\]
Note that $\vf_r(y)$ has at most $1/r$ elements since $||\nu_y||\le 1$ for every $y\in L$.

\begin{lemma}\label{fvm:1}
Let $\eps>0$ and let $F\sub L $ be a closed set  such that $\osc_y(\theta,F)<\eps$ for $y\in F$.

\begin{itemize}
\item[(i)] If $U\sub K$ is open and $\vf_{r-\eps}(y)\sub U$ for some $y\in F$ then there is a neighbourhood $W$ of $y$ in $F$ such
that $\vf_r(z)\sub U$ for every $z\in W$.
\item[(ii)] $\overline{\vf_r[F]}\sub \vf_{r-\eps}[F]$.

\end{itemize}
\end{lemma}

\begin{proof}
As $\vf_{r-\eps}(y)\sub U$, for any $x\in K\sm U$ we have $\big|\nu_{y}(\{x\})\big|< r-\eps$  and therefore there is an open set $U_x\ni x$ such that 
$|\nu_y|(\ol{U_x})<r-\eps$. There are $x_i$, $i\le i_0$, such that 
the sets   $U_i=U_{x_i}$ form a finite cover of $K\sm U$. 
Let 
\[\eta=\min\{r-\eps-|\nu_y|(\ol{U_i}):i\le i_0\}.\]

Using Lemma \ref{pr:1}(ii) we can find a set $W\ni y$ open in $F$ and  such that if $z\in W$ then
\[|\nu_z|(K\sm\ol{U_i})>|\nu_y| K\sm\ol{U_i})-\eta,\]
 for every $i\le i_0$; by our assumption on $F$ we can also demand that $||\nu_z||<||\nu_y||+\eps$ for $z\in W$.

Take any $z\in W$ and $x\in K\sm U$. Then $x\in U_i$ for some $i\le i_0$ and hence
\[|\nu_z|(\{x\})\le |\nu_z|(\ol{U_i}) = |\nu_z|(K)- |\nu_z|(K\sm \ol{U_i})<\]
\[<|\nu_y|(K)+\eps - |\nu_y|(K\sm \ol{U_i}) +\eta=|\nu_y|(\ol{U_i})+\eta+\eps\le r.\]
Hence $\vf_r(z)\sub U$ for $z\in W$, and this shows (i).

To check (ii) suppose that $x\notin \vf_{r-\eps}(y)$ for any $y\in F$. Then for every
$y\in F$ there is $U_y\ni x$ such that $\vf_{r-\eps}(y)\sub K\sm\ol{U_y}$. By (i)
$\vf_r(z)\sub K\sm\ol{U_y}$ for $z$ from some set $V_y\ni y$ open in $F$. Take a finite cover
$V_{y_j}$, $j\le j_0$ of $F$ and let $U=\bigcap_{j\le j_0} U_{y_j}$. Then $U$ is disjoint from
$\vf_r[F]$, so $x\notin \ol{\vf_r[F]}$, as required. 
\end{proof}

\begin{lemma}\label{fvm:3}
If $U\sub K$ is a nonempty open set then there are $\eps>0$, a compact set $K_1\sub U$ having nonempty interior, 
a closed set $F\sub L$, and $s>0$ such that 

\begin{enumerate}
\item $K_1\sub \vf_s[F]$ and $K_1\cap \vf_{3s/2}[F]=\emptyset$;
\item every $x\in \wn K_1$ has a neighbourhood $H$ such that $\card(\vf_{s-2\eps}(y)\cap \ol{H})\le 1$ 
for every $y\in F$.
\end{enumerate}
\end{lemma}

\begin{proof}
Let $\eps=m/20$ and let $F$ be as in Theorem \ref{i:4}. Then
\[r_0=\sup \{r>0: U\cap \wn{ \ol{\vf_r[F]}}\neq\emptyset\},\]
satisfies $r_0\ge m-2\eps$ by \ref{i:4}.

If we now take $s$ such that $s+\eps<r_0<3s/2$ then 
\[ U\cap \left( \ol{\vf_{s+\eps}[F]}\sm  \ol{\vf_{3s/2}[F]} \right) ,\]
contains a compact set $K_1$ with nonempty interior. Thus
$K_1\cap \vf_{3s/2}[F]=\emptyset$ and $\vf_s[F]\supseteq \ol{ \vf_{s+\eps}[F]}\supseteq K_1$ by Lemma \ref{fvm:1}(ii).

To verify the second part, fix $x\in L_1$, take any $y\in F$ and choose
$U_y\ni x$ such that $|\nu_y|(\ol{U_y})<(3/2)s$.
There is open $V_y\ni y$ such that for every $z\in V$
\[ |\nu_y|(K)>|\nu_z|(K)-\eps;\]
\[ |\nu_z|(K\sm\ol{U_y}) > |\nu_y|(K\sm\ol{U_y})-\eps.\]
Indeed, the first requirement can be fulfilled by the property of $F$ while the second by Lemma \ref{pr:1}(ii).
It follows that for any $z\in V_y$ we have 
\[|\nu_z|(\ol{U_y})=|\nu_z|(K)-|\nu_z|(K\sm\ol{U_y})\le\]
\[ \le |\nu_y|(K)-|\nu_y|(K\sm\ol{U_y})+2\eps=|\nu_y|(\ol{U_y})+2\eps<(3/2)s+2\eps.\]
Note that $(3/2)s+2\eps<2(s-2\eps)$, which is equivalent to $12\eps<s$: indeed, $3/2s>m-2\eps$ so
$s>2/3m-4/3\eps=2/3 \cdot 20\eps-4/3\eps=12\eps$.
Therefore
$\vf_{s-\eps}(y)$ cannot intersect $\ol{U_y}$ at two points.

Take a finite set $F_0\sub F$ such that the sets  $V_y$, for $y\in F_0$ form a cover of $F$.
Then $H=\bigcap_{y\in F_0} U_y$ is  as required.
\end{proof}

\section{Results}

Now we are ready to state and prove our main result. Recall that a family of nonempty open subsets
$\VV$ is a $\pi$-base of $K$ if for every nonempty open $U\sub K$ there is $V\in\VV$ such that $V\sub U$.
 
\begin{theorem}\label{rq:1} 
Let $K$ and $L$ be compact spaces such that $C(K)$ is isomorphic to $C(L)$. 
Then $K$ has a $\pi$-base $\VV$ such that for every $V\in\VV$, $\ol{V}$ is a continuous image of some compact subspace
of $L$.
\end{theorem} 

\begin{proof} 
Given nonempty open $U\sub K$,  we take $K_1$, $F$, $s$ and $\eps$ as in Lemma \ref{fvm:3}.
Let $K_2=\ol{H}\sub K_1$, where $H\neq \emptyset$ is as in part (2) of Lemma \ref{fvm:3}.
Put 
\[Y=\{y\in F: \vf_s(y)\cap K_2\neq\emptyset\}.\]
Then $\vf_{s-\eps}(y)\cap K_2\neq\emptyset$
for every $y\in\ol{Y}$ by Lemma \ref{fvm:1}(i).

We define $h:\ol{Y}\to K_2$ so that $h(y)$ is the unique point
in $K_2\cap \vf_{r-\eps}(y)$. Then $h$ maps $\ol{Y}$ onto $K_2$ so it remains to
check that $h$ is continuous.

Let the set $C\sub K_2$ be closed. Then $h^{-1}[C]=A$, where
\[A=\{y\in\ol{Y}: \vf_{r-\eps}(y)\cap C\neq\emptyset\},\]
and $A$ is closed. Indeed, if $y\in \ol{Y}\sm A$, i.e.\ $\vf_{r-\eps}(y)\cap C=\emptyset$
then $\vf_{r-2\eps}(y)\cap C=\emptyset$ as well by \ref{fvm:3}(2), and it follows from Lemma \ref{fvm:1}
that $\vf_{r-\eps}(z)\cap C=\emptyset$ for all $z$ from some neighbourhood of $y$.
\end{proof}

Of course the above theorem gives no information for spaces $K$ having dense sets of
isolated points. On the other hand, the result has the following consequences.


\begin{corollary}\label{rq:3}
Given a cardinal number $\kappa$ and a compact space $L$, if $C[0,1]^\kappa$ is isomorphic to $C(L)$ then $L$ maps continuously onto $[0,1]^\kappa$.
\end{corollary}

\begin{proof}
Clearly, every nonempty open subset of  $[0,1]^\kappa$ contains a subset homeomorphic to the whole space.
Hence if $C(K)\sim C(L)$ then Theorem \ref{rq:1} implies that there is a compact subspace $L_1\sub L$ and
a continuous surjection $h: L_1\to [0,1]^\kappa$; in turn such $h$ can be extended
to a continuous mapping $L\to [0,1]^\kappa$ by the Tietze extension theorem.
\end{proof}

Recall that a topological space $X$ is said to be {\em homogeneous} if for every $x,x'\in X$ there is
a homeomorphism $f$ of $X$ onto itself such that $f(x)=x'$.

\begin{corollary}\label{rq:4}
If $L$ is Corson compact and $C(K)\sim C(L)$ for some compact $K$ 
then $K$ has a $\pi$-base of sets
with Corson compact closures. In particular, $K$ is Corson compact itself whenever
$K$ is homogeneous.
\end{corollary}

\begin{proof}
The first assertion follows from Theorem \ref{rq:1} since the class of Corson compacta is closed under taking compact subspaces and continuous images. 

If $K$ is homogeneous then it follows that $K$ can be covered by a finite family
$\{V_i: i\le i_0\}$ of open sets where $\overline{V_i}$ is Corson compact for
every $i$. Then a disjoint union $K'=\bigoplus_{i\le i_0} \overline{V_i}$ is Corson compact
and so is $K$ since $K$ is a continuous image of $K'$.
\end{proof}

Let us note that in Theorem \ref{rq:1} cannot 
be strengthen by replacing a $\pi$-base with a base.
This can be seen using a result due to Okunev
\cite{Ok05} showing that isomorphisms of $C(K)$ spaces do not preserve the Frechet property; cf.\ \cite{Pl13}, section 5.

\end{document}